\documentclass[11pt]{amsart}
\usepackage{url,graphicx,tabularx,array,geometry}
\usepackage{graphicx}
\usepackage{amssymb}
\usepackage{fancyhdr}
\usepackage{wrapfig}
\usepackage{epstopdf}
\usepackage{amsmath}
\usepackage[labelfont=bf]{caption}
\usepackage[hidelinks]{hyperref}
\usepackage{amsthm}
\usepackage{bbm}
\usepackage[T1]{fontenc}
\usepackage{yfonts}
\usepackage{harpoon}
\usepackage[all]{xy}
\usepackage{tikz}
\usetikzlibrary{decorations.markings}
\tikzstyle{vertex}=[circle, draw, inner sep=0pt, minimum size=6pt]
\usepackage{hyperref}

\DeclareMathOperator{\Aut}{Aut}

\DeclareMathOperator{\Inn}{Inn}
\DeclareMathOperator{\Fit}{Fit}
\DeclareMathOperator{\PSL}{PSL}

\newcommand{\cP}{\mathcal{P}}
\newcommand{\cS}{\mathcal{S}}
\newcommand{\cQ}{\mathcal{Q}}
\newcommand{\cR}{\mathcal{R}}
\newcommand{\cI}{\mathcal{I}}

\newcommand{\cC}{\mathcal{C}}
\DeclareMathOperator{\cd}{cd}
\DeclareMathOperator{\eo}{eo}
\DeclareMathOperator{\cs}{cs}

\newcommand{\nor}{\vartriangleleft}
\newcommand{\Z}{\mathbb{Z}}

\renewcommand{\gcd}[1]{\operatorname{gcd}\left\{{#1}\right\}}

\makeatletter
\def\imod#1{\allowbreak\mkern10mu({\operator@font mod}\,\,#1)}
\makeatother
\theoremstyle{plain}

\newtheorem{theoremN}{Theorem}[section]

\newtheorem{corollaryN}[theoremN]{Corollary}
\newtheorem{lemmaN}[theoremN]{Lemma}

\theoremstyle{definition}
\newtheorem*{definition}{Definition}
\theoremstyle{remark}

\newtheorem*{remark}{Remark}
\newtheorem*{note}{Note}

\newcommand{\tightoverset}[2]{%
  \mathop{#2}\limits^{\vbox to -.5ex{\kern-0.75ex\hbox{$#1$}\vss}}}

\newcommand{\tikzoverset}[2]{%
  \tikz[baseline=(X.base),inner sep=0pt,outer sep=0pt]{%
    \node[inner sep=0pt,outer sep=0pt] (X) {$#2$};
    \node[yshift=1pt] at (X.north) {$#1$};
}}

\newcommand{\overharpo}{\tikzoverset{\rightharpoonup}}

\begin{document}
\title[Prime Graphs of Several Classes of Finite Groups]{The Prime Graphs of Some Classes of Finite Groups}

\author[Florez, Higgins, Huang, Keller, Shen, and Yang]{Chris Florez, Jonathan Higgins, Kyle Huang, Thomas Michael Keller, Dawei Shen, and Yong Yang}

\address{Christopher Florez, Department of Mathematics, David Rittenhouse Lab, University of Pennsylvania, 209 South 33rd Street, Philadelphia, PA 19104-6395, USA}
\email{cflorez@sas.upenn.edu}

\address{Jonathan Higgins, Mathematics and Computer Science Department, Wheaton College, 501 College Ave, Wheaton, IL 60187, USA}
\email{jahiggins11@icloud.com}

\address{Kyle Huang, Mathematics Department, University of California-Berkeley, 2227 Piedmont Avenue, Berkeley, CA 94709, USA}
\email{kyle.huang@berkeley.edu}

\address{Thomas M. Keller, Department of Mathematics, Texas State University, 601 University Drive, San Marcos, TX 78666-4616, USA}
\email{keller@txstate.edu}

\address{Dawei Shen, Department of Mathematics and Statistics, Washington University in St. Louis, 1 Brookings Drive, St. Louis, MO 63105, USA}
\email{shen.dawei@wustl.edu}

\address{Yong Yang, Department of Mathematics, Texas State University, 601 University Drive, San Marcos, TX 78666-4616, USA}
\email{yang@txstate.edu}


\begin{abstract}
In this paper we study prime graphs of finite groups. The \textit{prime graph} of a finite group $G$, also known as the Gruenberg-Kegel graph, is the graph with vertex set \{primes dividing $|G|$\} and an edge $p$-$q$ if and only if there exists an element of order $pq$ in $G$.
In finite group theory, studying the prime graph of a group has been an
important topic for the past almost half century. Only recently prime graphs of solvable groups
have been characterized in graph theoretical terms only. 
In this paper, we continue this line of research and give complete characterizations of several classes of groups, including groups of square-free order, metanilpotent groups, groups of cube-free order, and, for any $n\in \mathbb{N}$, solvable groups of $n^\text{th}$-power-free order.\\
We also explore the prime graphs of groups whose composition factors are cyclic or $A_5$ and draw connections to a conjecture of Maslova. We then propose an algorithm that recovers the prime graph from a dual prime graph.
\end{abstract}

\maketitle
\section{Introduction}

The topic of this paper is prime graphs of finite groups. The prime graph of a finite group is the graph whose vertices are the prime numbers
dividing the order of the group, with two vertices being linked by an edge if and only if their 
product divides the order of some element of the group. Prime graphs were introduced by
Gruenberg and Kegel in the 1970s to investigate certain cohomological questions associated with integral representations of finite groups and have been an object of constant study ever since.
They were among the first graphs assigned to groups. This idea of representing group theoretical
data via graphs and describing them via graph theoretical notions is proved to be so successful 
that today there is a myriad of graphs (e.g. character degree graphs, conjugacy class size
graphs, etc.) and a whole industry of exploring them. For this reason, today prime graphs are
often referred to as Gruenberg-Kegel graphs.\\

The focus in the study of prime graphs has been on simple groups for a long time, so that these graphs today are well understood (see \cite{kondratev1989prime} and \cite{WILLIAMS1981487}).
In contrast, the main result of \cite{GRUBER2015397}, somewhat surprisingly, is a purely graph theoretical
characterization of prime graphs of solvable groups: A (simple) graph is the prime graph
of a finite solvable group if and only if its complement is triangle-free and 3-colorable.
This opened up new ways of studying prime graphs; for example,
simple groups whose prime graph is that of a finite solvable group have been characterized in
\cite{GM2018}. It also inspired similar characterization efforts for other graphs; For example, Dolfi, Pacifici, Sanus, and Sotomayor \cite{DPSS2020} have studied similar questions related to the
prime graphs of conjugacy class sizes and also found nice properties.\\

The purpose of this paper is to take the characterization of prime graphs of finite solvable groups in two directions: Specialization and generalisation. Regarding the former,
we ask what the prime graphs look like for classes of groups between nilpotent and solvable.
For nilpotent groups, the prime graph is obviously complete, and for solvable groups we have
the above characterization. Here we characterize the prime graphs for, among others, the
following classes of solvable groups: groups of square-free order and metanilpotent 
groups (Section 3), solvable groups of cube-free order, and, more generally, for any $n\in \mathbb{N}$, solvable groups of $n^\text{th}$-power-free order (Section 4).\\

 As to the generalization aspect, we take the first baby steps towards giving characterizations
of prime graphs of classes of groups that are not necessarily solvable any longer.
Groups of square-free order are always solvable, but we will actually characterize the
prime graphs of arbitrary groups of cube-free order (Section 5). Moreover, we will study the prime graphs
of a class of groups "minimally away" from being solvable, namely groups all of whose
composition factors are either cyclic or isomorphic to $A_5$ (Section 6). It will become clear that
leaving the realm of solvable groups is quite a challenging task in this area.\\

Finally we will shift our attention to the "dual prime graph", also known as common divisor graph,
whose vertices are the element orders of the group with two vertices being connected by an
edge if and only if they have a common divisor greater than 1 (Section 7). We present an algorithm that recovers the prime graph from a dual prime graph and actually present the result in a
much more general setting.\\

In Section $2$, we will briefly review what solvability implies about prime graphs, most of which is taken from \cite{GRUBER2015397}. \\

%\bigskip

%The question of how group theoretic properties influence the conditions on their prime graphs %remains, and this is what we will mainly investigate in this paper. We are going to pursue along %this direction further and explore how the graph theoretic properties on the prime graphs react %when we perturb the group theoretic properties around solvability. For the most part, we will %work with conditions that are stronger than solvability, since this allows us to adopt the %setting of solvable groups and apply results from \cite{GRUBER2015397}. That being said, we will %explore a condition slightly weaker than solvability towards the end. At the end of the paper, . \bigskip

Before we begin, we introduce conventions and notations used throughout this paper. All graphs will be simple and all groups will be finite. Let $\pi(G)$ denote the set of all prime divisors of $|G|$, $\pi$ usually denotes some subset of $\pi(G)$, and $\pi'$ denotes $\pi(G)\setminus \pi$ unless stated otherwise. Since a prime graph has vertex set $\pi(G)$ and all graphs we consider are prime graphs, we use $\pi(G)$ and $V(F)$ interchangeably.\bigskip

In an undirected graph $F$, we denote an edge connecting $p$ and $q$ by $p\text{-}q$. In a directed graph $\overharpo{F}$, we denote an edge from $p$ to $q$ by $p\to q$. For $\pi\subset V(F)$, let $F[\pi]$ denote the induced subgraph of $F$ on the vertex set $\pi$. When we refer to a path in a directed graph, it is assumed that the path is directed. We refer to paths on $n+1$ vertices with $n$ edges as $n$-paths. Let $F\setminus \{p\text{-}q\}$ be the graph obtained by removing the edge between $p$ and $q$ from $E(F)$, if there exists one. We refer to graphs obtained by taking an induced subgraph and then removing edges as subgraphs of the original graph. \bigskip

We adopt the ATLAS notation and write $G=N.M$ if $N$ is a normal subgroup of $G$ such that $G/N\simeq M$. We write the Fitting subgroup of $G$ as $F(G)=F_1(G)$. We build the Fitting series $\{F_k\}$ of $G$ by taking $F_{k+1}/F_{k}=F(G/F_{k})$. Unless stated otherwise, $P$ will be an arbitrary Sylow $p$-subgroup of $G$. We use $H_\pi$ to denote a Hall $\pi$-subgroup of $G$. And when $\pi$ consists of two (or resp., $3$) primes $p$ and $q$ (and resp., $r$), we simply write $H_{pq}$ (or resp., $H_{pqr}$).

\section{Preliminary Terminology and Results}
We first briefly go through the main results from \cite{GRUBER2015397} to establish the setting.
\begin{definition}
A group $G$ is a \emph{$2$-Frobenius group} if $F_2$ is Frobenius with kernel $F_1$ and $G/F_1$ is Frobenius with kernel $F_2/F_1$. We refer to the Frobenius kernel of $G/F_1$ as the \emph{upper kernel} of $G$ and the Frobenius kernel $F_1$ as the lower kernel of $G$.
\end{definition}
Thus, the primes dividing $|F_2:F_1|$ must be disjoint from the primes dividing $|G:F_2|$ and $|F_1|$, since the Frobenius kernel and the Frobenius complement must have coprime order.
\begin{definition}\cite{GRUBER2015397} We categorize Frobenius groups and $2$-Frobenius groups of order divisible by at most $3$ primes as follows.
\begin{enumerate}
    \item $G$ is said to be Frobenius of type $(p,q)$ if its Frobenius complement is a $p$-group and its Frobenius kernel is a $q$-group.
    \item $G$ is said to be $2$-Frobenius of type $(p,q,r)$ if $F_1$ contains a Sylow $p$-subgroup, $F_2/F_1$ is a Sylow $q$-subgroup, and $G/F_2$ is an $r$-group.
    \item In particular, $G$ is said to be $2$-Frobenius of type $(p,q,p)$ if $G/F_2$ and $F_1$ are both $p$-groups and $F_2/F_1$ is a q-group.

\end{enumerate}
\end{definition}

By \cite[Theorem A]{WILLIAMS1981487}, any solvable group with a disconnected prime graph is Frobenius or $2$-Frobenius. So for $G$ solvable, for any edge $p\text{-}q$ in $\overline{\Gamma}$, $\Gamma[\{p,q\}]$, the prime graph of $H_{pq}$, is disconnected. This gives us that $H_{pq}$ is Frobenius or $2$-Frobenius. Since the Frobenius kernel and the Frobenius complement must have coprime order, $H_{pq}$ must be Frobenius of type $(p,q)$ or $(q,p)$ or 2-Frobenius of type $(p,q,p)$ or $(q,p,q)$.

\begin{definition}\cite{GRUBER2015397} 
For a finite solvable group $G$, we define an orientation of $\overline{\Gamma}(G)$ as follows.  Let $p\text{-}q$ be any edge in $\overline{\Gamma}(G)$. if $H_{pq}$ is Frobenius of type $(p,q)$ or $2$-Frobenius group of type $(p,q,p)$, we direct $p\text{-}q$ as $p\to q$. We call this orientation of $\overline{\Gamma}(G)$ the \emph{Frobenius digraph} of $G$, denoted $\overharpo{\Gamma}(G)$.
\end{definition}

\begin{lemmaN}\cite[Lemma 2.4]{GRUBER2015397}\label{pqr}
Let $G$ be solvable. If $r\to q\to p$ is a $2$-path in $\overharpo{\Gamma}(G)$, then $H_{pqr}$ is $2$-Frobenius of type $(p,q,r)$.
\end{lemmaN}

\begin{theoremN}\cite[Corollary 2.7]{GRUBER2015397}
The Frobenius Digraph of a solvable group does not contain a directed $3$-path
\end{theoremN}
\begin{lemmaN}[Lucido's Three Primes Lemma, \cite{lucido1999diameter}]
Let $G$ be a solvable group. If $p,q,r$ are distinct primes dividing $|G|$, $G$ contains an element of order the product of two of these primes.
\end{lemmaN}
\begin{theoremN}[Gallai-Hasse-Roy-Vitaver] \label{graphtheory}
A graph is $k$-colorable if and only if there exists an orientation that does not contain a $k$-path.
\end{theoremN}

\begin{definition}
The canonical orientation of a graph $F$ with a $3$-coloring by Red, Green and Blue is obtained by directing all edges
\begin{itemize}
    \item from red to green
    \item from green to blue
    \item from red to blue
\end{itemize}
Notice that this orientation contains no directed $3$-paths.
\end{definition}

\begin{theoremN}
A graph is $2$-colorable if and only if it is triangle free.
\end{theoremN}

Finally, we have the main theorem of \cite{GRUBER2015397}. 
\begin{theoremN}\cite[Theorem 2.10]{GRUBER2015397} \label{PG of Solvable Groups}
A graph $F$ is isomorphic to the prime graph of some solvable group if and only if $\overline{F}$ is $3$-colorable and triangle free. 
\end{theoremN}
\begin{remark}
The forward direction is easily obtained by combining the above results. If $F$ is the prime graph of some solvable group, $\overline{F}$ is triangle free by Lucido's Three Primes Lemma. The Frobenius Digraph gives an orientation that is $3$-path-free, so $\overline{F}$ is $3$-colorable by the Gallai-Hasse-Roy-Vitaver theorem.\bigskip

The backward direction involves constructing a solvable group for each graph whose complement is $3$-colorable and triangle free. Variations of this construction will be used later on in this paper multiple times. In particular, we are able to construct solvable groups of odd order for any graph whose complement is $3$-colorable and triangle free by avoiding $2$ when assigning primes for each vertex.
\end{remark}

Notice that in the above Theorem \ref{PG of Solvable Groups}, when constructing a solvable group for a given graph with a triangle free and $3$-colorable complement, we would often end up with a group of order divisible by large powers of primes. It is natural to conjecture that when we restrict the highest power of primes that could appear in the group order, we must have stronger properties on the corresponding prime graphs. This inspired us to study the prime graphs of groups of square-free, cube-free, and, more generally, $n^{\text{th}}$-power-free order. It turns out that our first approach to studying prime graphs of square-free order groups can be generalized for prime graphs of metanilpotent groups. But in the general case for groups of $n^{\text{th}}$-power-free order, we must adopt a different strategy, as we shall see in the next section.

%Towards the end of the paper, we explore the case where the condition on the groups is weaker than solvability.

\section{Groups between Metanilpotent and Square-free-order}
Groups of square-free order are metanilpotent, and thus solvable. Notice that if we only do the first part of the construction in Theorem \ref{PG of Solvable Groups} for all vertices that are not the end of a directed $2$-path, we will end up with a group of square-free order. Therefore, for each bipartite graph, there is a group of square-free order that admits it as the complement prime graph. In fact, we will see that it is exactly the condition on the complement prime graphs of groups of square-free order. And it can be generalized to the following theorem.
\begin{theoremN}\label{first char.}
The following two statements hold.
\begin{enumerate}
    \item Let $G$ be a metanilpotent group and $F$ be its prime graph. Then $\overline{F}$ is bipartite.
    \item If $F$ is a graph such that $\overline{F}$ is bipartite, then there exists a group of square-free order whose prime graph is isomorphic to $F$.
\end{enumerate}
\end{theoremN}
\begin{proof}
The proof involves two parts. We will first prove that the complement prime graph of a metanilpotent group is bipartite, and for each given bipartite graph, we will construct a group of square-free order that admits it as the complement prime graph.\bigskip

If $G$ is metanilpotent, then $G$ is solvable, so Hall subgroups exist. If there exists a directed $2$-path $r\to q\to p$ in the Frobenius Digraph, then the Hall $\{p,q,r\}$-subgroup $H$ is a $2$-Frobenius group by Lemma \ref{pqr}. Therefore, $H/F(H)$ is Frobenius. On the other hand, $H$ is metanilpotent because $G$ is metanilpotent. So $H/F(H)$ is nilpotent. This is a contradiction, since Frobenius groups cannot be nilpotent. Therefore, there does not exist a directed 2-path in the Frobenius Digraph. So $\overline{\Gamma}(G)$ is $2$-colorable by the Gallai-Hasse-Roy-Vitaver theorem.\bigskip

Conversely, assume that $\overline{F}$ is 2-colorable. We repeat the first part of the construction given in \cite[Theorem 2.8]{GRUBER2015397}. Take any 2-coloring of it using red and blue and direct the edges from red to blue. By construction, the resulting orientation contains no directed 2-paths. Now, we find a group $G$ of square-free order whose prime graph is isomorphic to $F$.\bigskip

Let $\mathcal{O}$ be the set of vertices in $\overline{F}$ with in-degree $0$ and non-zero out-degree and $\mathcal{I}$ be all the other vertices. Then all vertices in $\mathcal{I}$ have $0$ out-degree, because otherwise there will be a directed 2-path. Let $|\mathcal{O}|=m$ and $|\mathcal{I}|=n$.\bigskip

Let $\mathcal{P}=\{p_j\in \mathbb{P}: j=1,\dots,m\}$ be a set of distinct primes and $p$ be their product. By Dirichlet's theorem, we can pick $\mathcal{Q}=\{q_k\in \mathbb{P}:k=1,\dots,n\}$ a set of distinct primes such that $q_k\equiv1(mod\ p)$ for each $q_k$. Define a directed graph $\overline{\Lambda}$ with vertex set $\mathcal{P} \bigcup \mathcal{{Q}}$ and edge set defined by a fixed graph  isomorphism $\overline{F}\to \overline{\Lambda}$ mapping vertices in $\mathcal{O}$ and $\cI$ to primes in $\mathcal{P}$ and $\cQ$, respectively. Let $T=C_{p_1}\times \cdots \times C_{p_m}$ and $U=C_{q_1}\times \cdots \times C_{q_n}$. Since $p_j \mid q_k-1$ for each $j$ and $k$, we can define a semidirect product $K=U\rtimes T$ by letting $C_{p_j}$ act fixed-point-freely on $C_{q_k}$ if $p_j q_k\in \overline{\Lambda}$ and trivially otherwise. Then each Hall $\{p_i,q_j\}$-subgroup of $K$ is a Frobenius group if $p_j q_k\in \overline{\Lambda}$ and a direct product otherwise. So $F$ isomorphic to $\Lambda$, the prime graph of $K$, which is of square-free order.
\end{proof}

\begin{remark}
Let $\mathcal{C}$ be any one of the following classes of groups.
\begin{itemize}
    \item Metanilpotent groups
    \item Groups whose commutator subgroup is nilpotent
    \item Supersolvable groups
    \item Metacylic groups
    \item Metabelian groups
    \item Z-groups
    \item Groups of square-free order
\end{itemize}
Then a graph $F$ is isomorphic to the prime graph of some group in $\mathcal{C}$ if and only if $\overline{F}$ is bipartite. This is because any group of square-free order is in $\mathcal{C}$ and any group in $\mathcal{C}$ is metanilpotent.

We will see in Section $5$ that groups of cube-free and odd order are metanilpotent. By the above theorem, prime graphs of these groups have a bipartite complement. It turns out that every graph with a bipartite complement is isomorphic to  the prime graph of some group of cube-free and odd order. Having cube-free order is guaranteed by the above theorem, and it requires proof why we can always make the group order odd. Note that for any graph $F$ with bipartite complement, in our construction of a group $G$ of square-free order such that the prime graph of $G$ is isomorphic to $F$, we are free to choose the primes in $\pi(G)$ as long as certain number-theoretic conditions are satisfied. And we can always find primes fulfilling these conditions such that they avoid any given set of prime numbers. In particular, we can choose to avoid the prime number $2$ in $\pi(G)$. This will help us characterize prime graphs of groups of cube-free order in Section $5$.

\end{remark}

\section{Solvable groups of $n^{th}$-power-free order}
We first show the following lemma, which explains why the $p$-group at the bottom of a $2$-Frobenius group of type $(p,q,r)$ will have to be large.
\begin{lemmaN}\label{big order bottom}
If $H$ is a 2-Frobenius group of type $(p,q,r)$, then $H$ contains a subgroup of order $p^k$, where $r$ divides $k$.
\end{lemmaN}
\begin{proof}
Let $F_1=\Fit(H)$ and $\Fit(H/F_1)=F_2/F_1$. $H/F_1$ acts on $F_1$ by conjugation. Take $P_1$ to be the unique Sylow $p$-subgroup of $F_1$. Then $H/F_1$ restricts to an action on $P_1$, since the taking a conjugation sends $p$-elements to $p$-elements. Since $\Omega_1(P_1)$ is characteristic in $P_1$, $\Omega_1(P_1)$ is invariant under the group action. So now we have that $H/F_1$ acts on $\Omega_1(P_1)$ elementary abelian. All nontrivial elements in $\Omega_1(P_1)$ has order $p$, so $\Omega_1(P_1)$ is elementary abelian and thus can be viewed as a vector space over $\mathbb{F}_p$. Then the group action makes $\Omega_1(P_1)$ an $\mathbb{F}_p(G/F_1)$-module. Also, since $F_2/F_1$ acts fixed-point-freely on $F_1$, it acts fixed-point-freely on $\Omega_1(P_1)$. Note that $F_2/F_1$ is a $q$-group, so $p\nmid |F_2/F_1|$. By \cite[Lemma 0.34]{manz_wolf_1993}, $\Omega_1(P_1)$ has dimension divisible by $r$ over $\mathbb{F}_p$.
\end{proof}
\begin{corollaryN}\label{small top}
If $|G|$ is $n^\text{th}$-power-free and $\pi\subset \pi(G)$ such that $\prod_{r\in\pi}r\geq n$, then any vertex in the Frobenius digraph cannot be the end of a $2$-path starting from each of the elements in $\pi$.  
\end{corollaryN}
\begin{proof}
Suppose that $p$ is the end of a $2$-path starting from each $r\in \pi$. Let $|Z(P)|=p^k$. By Lemma \ref{big order bottom}, $r\mid k$ for each $r\in \pi$. Therefore, $\prod_{r\in\pi}r\mid k$. So $k\geq \prod_{r\in\pi}r \geq n $. This contradicts with $|P|$ being $n^\text{th}$-power-free.
\end{proof}

\begin{theoremN}\label{nth power free}
A graph $F$ is isomorphic to the prime graph of a solvable group of $n^\text{th}$-power-free order if and only if $\overline{F}$ satisfies the following conditions.
\begin{enumerate}
    \item $\overline{F}$ is triangle-free.
    \item There exists a $3$-coloring of $\overline{F}$ by Red, Green, and Blue and a way to label each red vertex by a distinct prime number such that for any subset $\pi\subset \pi(G)$ satisfying $\prod_{p\in \pi} p  \geq n$, we have that in the canonical orientation, no blue vertex is simultaneously the end of directed $2$-paths starting from each of the red vertices in $\pi$.
    
\end{enumerate}
\end{theoremN}
\begin{proof}
($\Rightarrow$) If $F$ is the prime graph of a solvable group, then its complement $\overline{F}$ is triangle-free and $3$-colorable \cite[Theorem 2.8]{GRUBER2015397}. \cite[Theorem 2.8]{GRUBER2015397} also showed that we can direct the edges in $\overline{F}$ to form the Frobenius Digraph, which is $3$-path-free. We can then color the vertices by red, green, and blue such that any $2$-path is from red to green to blue. For any $\pi\subset \pi(G)$ a collection of distinct primes whose product is at least $n$, any vertex in the Frobenius digraph cannot simultaneously be the end of a $2$-path starting from each of the elements in $\pi$ by Corollary \ref{small top}. The forward direction is proved.\bigskip

($\Leftarrow$) Direct the edges in $\overline{F}$ and label each vertex in $F$ by a distinct prime number such that the required conditions are satisfied. Follow the same construction as in \cite[Theorem 2.8]{GRUBER2015397}. For each vertex $q$ that is not the end of a directed 2-path, $q$ appears exactly once as a divisor of the group order. For any vertex $p$ that is the end of some directed $2$-path, $p$ appears exactly $\prod_{r\in\pi}r$ times in the group order, where $\pi$ is the set of vertices starting from which there is a directed $2$-path ending at $p$. By our condition, $\prod_{r\in\pi}r\leq n-1$. So the resulting group is of $n^\text{th}$-power-free order.
\end{proof}
We rephrase the graph theoretic conditions in the above theorem as the following.
\begin{corollaryN}\label{nth power free simplified}
A graph $F$ is isomorphic to the prime graph of a solvable group of $n^\text{th}$-power-free order if and only if $\overline{F}$ satisfies the following conditions.
\begin{enumerate}
    \item $\overline{F}$ is triangle-free.
    \item There exists a $3$-coloring of $\overline{F}$ by Red, Green, and Blue such that we can label the red vertices by the first $m$ primes $2=p_1<3=p_2<\cdots <p_m$ such that
\begin{enumerate}
    \item $p_m<n$.
    \item Let $\pi$ be any subset of $\{p_1,p_2,\cdots, p_m\}$ such that $\prod_{p\in \pi} p  \geq n$. Then in the canonical orientation, no blue vertex is simultaneously the end of a directed $2$-path starting from each of the red vertices in $\pi$, 
\end{enumerate}
\end{enumerate}
\end{corollaryN}
\begin{proof}
Note that the second condition in this corollary is a refinement of the second condition stated in Theorem \ref{nth power free}. Therefore, it suffices to only prove the forward direction. Note that the smaller the primes we use to label the red vertices, the easier the above conditions are satisfied. Therefore, it suffices to only consider using the first $m$ primes, where $m$ is the number of red vertices. Furthermore, we can assume without loss of generality that there all red vertices are the start of some directed $2$-path. This is because otherwise, we can change the color of a red vertex $r$ that is not the start of a directed $2$-path into green, and change the color of each green vertex adjacent to $r$ to be blue. There will be no originally blue vertices connected to any of the above originally green vertices, because otherwise, we have a $2$-path starting from $a$. Therefore, we still have a $3$-coloring after the change. Since each red vertex is the start of some $2$-path, we cannot use a prime that is at least $n$ to label it by Theorem \ref{nth power free}. Therefore, the largest prime to be used must be less than $n$.
\end{proof}
\begin{remark}
The graph theoretic condition in the above corollary is a much easier condition to check than Theorem \ref{nth power free}. We can run the following algorithm on $\overline{F}$.
\begin{enumerate}
    \item Check that it is triangle-free.
    \item Find the biggest integer $m$ such that the $m^{\text{th}}$ prime is less than $n$.
    \item For each $k\leq m$, go through each way of picking $k$ vertices for the red color and check if the remaining graph is bipartite. If it is, further check if there exists a $2$-coloring of the rest of the vertices by Green and Blue and a way to label the $k$ red vertices by the first $k$ primes such that it satisfies the condition in Corollary \ref{nth power free simplified}.
\end{enumerate}
\end{remark}
The above theorem and corollary may seem complicated. However, we established an if-and-only-if condition between group-theoretic conditions on the group and graph-theoretic conditions on the its prime graph. The graph-theoretic conditions can be further simplified when $n$ is small, as we will demonstrate later in the case when $n\leq 5$. When $n=2$, it gives the same result as in Theorem \ref{first char.}. The case when $n=3$ is useful to studying general (not necessarily solvable) groups of cube-free order, as we will show in the next section. And the graph theoretic conditions are the same for $n=4$ and $n=5$.\\
Next we recover the result in Section 3 on groups of square-free order as a consequence of the previous result.
\begin{corollaryN}\label{2nd power free}
$F$ is isomorphic to the prime graph of a group of square-free order if and only if $\overline{F}$ is $2$-colorable.
\end{corollaryN}
\begin{proof}
A group of square-free order is always solvable. Apply Corollary \ref{nth power free simplified}. First, assume that $F$ is the prime graph of a group of square-free order. Since no prime is less than $2$, no vertex can be colored red. So the graph is $2$-colorable, or bipartite. Conversely, bipartite graphs satisfy the graph theoretic conditions in Corollary \ref{nth power free simplified}.
\end{proof}
The above corollary coincides with Theorem \ref{first char.}.
\begin{corollaryN}\label{solvable cube free}
$F$ is isomorphic to the prime graph of a solvable group of cube free order if and only if $\overline{F}$ is triangle free and is $2$-colorable after removing one vertex. 
\end{corollaryN}
\begin{proof}
Apply Corollary \ref{nth power free simplified}. First, assume that $F$ is the prime graph of a group of cube-free order. Since $2$ is the only prime less than $3$, at most one vertex can be colored red. So the graph is triangle free and is $2$-colorable with $1$ vertex removed. Conversely, it is straightforward to see that these graphs satisfy the graph theoretic conditions in Corollary \ref{nth power free simplified}.
\end{proof}

\begin{corollaryN}\label{fourth power free}
The following are equivalent.
\begin{enumerate}
    \item $F$ is isomorphic to the prime graph of a solvable group of fourth-power-free order.
    \item $F$ is isomorphic to the prime graph of a solvable group of fifth-power-free order.
    \item $\overline{F}$ is $3$-colorable, triangle-free, and satisfies one of the following conditions.
    \begin{enumerate}
        \item $2$-colorable after removing one vertex
        \item There exists a $3$-coloring of $\overline{F}$ by Red, Green, and Blue such that there are exactly $2$ red vertices. Furthermore, in the canonical orientation, no vertex is the end of a directed $2$-path starting from both of the red vertices.
    \end{enumerate}

\end{enumerate}
\end{corollaryN}
\begin{proof}
Apply Corollary \ref{nth power free simplified}. Note that $2$ and $3$ are the only primes less than $4$ (or resp., $5$). Also, $2\times 3$ exceeds $4$ (or resp., $5$). 
\end{proof}

We now give an example of a graph that is triangle free and $3$-colorable but is not isomorphic to the complement of the prime graph of any solvable group of fifth-power-free order. This shows that by requiring our solvable group to have fifth-power-free order, we indeed end up with a more restrictive condition on its prime graph. It is also a nice demonstration of the algorithm in the remark under Corollary \ref{nth power free simplified}.

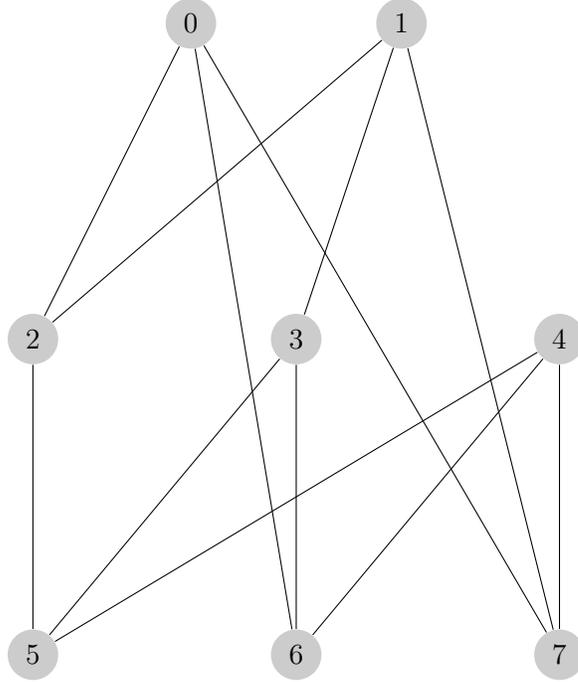
\begin{figure}
    \centering
\begin{tikzpicture}
  [scale=.7,auto=left,every node/.style={circle,fill=black!20}]
  \node (n0) at (3,12) {0};
  \node (n1) at (7,12)  {1};
  \node (n2) at (0,6)  {2};
  \node (n3) at (5,6) {3};
  \node (n4) at (10,6)  {4};
  \node (n5) at (0,0)  {5};
  \node (n6) at (5,0) {6};
  \node (n7) at (10,0) {7};

  \foreach \from/\to in {n0/n2, n0/n6, n0/n7, n1/n2, n1/n3, n1/n7, n2/n5, n3/n5, n3/n6, n4/n5, n4/n6, n4/n7}
    \draw (\from) -- (\to);
\end{tikzpicture}
\caption{A graph that is triangle free and $3$-colorable but is not isomorphic to the complement of the prime graph of any solvable group of fifth-power-free order.}
\end{figure}

\begin{theoremN}
Let $F$ be the graph in Figure 1. We have the following.
\begin{enumerate}
    \item $F$ is triangle free and $3$-colorable.
    \item $F$ is not isomorphic to the complement of the prime graph of any solvable group of fifth-power-free order.
\end{enumerate} 
\end{theoremN}
\begin{proof}
$\{0,1\}\bigcup\{2,3,4\}\bigcup\{5,6,7\}$ is a $3$-coloring. By going through each set of $3$ vertices, we see that it is triangle free.

Next, we show that this graph does not satisfy the last condition in Corollary \ref{fourth power free}. We will go through each possible way of taking $1$ vertex or a pair of vertices $[a,b]$ that is not adjacent to each other to be the red vertices.
\begin{enumerate}
    \item We use $[a,b]: (c,d)$ to denote that if $a$ and $b$ are the red vertices, $c$ and $d$ are, respectively, the vertex that is the end of a directed 2-path from both $a$ and $b$ in the two ways of coloring the rest of the graphs by Blue and Green.
    \item We use $[a,b]: (uvwxy)$ (or resp., $[a]: (uvwxy)$) to denote that if $a$ and $b$ are (or resp., $a$ is) the red vertices (or resp., vertex), the remaining graph is not 2-colorable due to the existence of the 5 cycle $u\to v\to w \to x \to y\to u$.
\end{enumerate}
Then, we have the following.\bigskip

$[0,1]: (5,4)$\ \ \ \ \ \ \ \ \ $[0,3]: (21745)$\ \ \ \ \ \ $[0,4]: (3,1)$\ \ \ \ \ \ \ \ \ $[0,5]: (17463)$ 

$[1,4]: (02536)$\ \ \ \ \ \ $[1,5]: (6,0)$\ \ \ \ \ \ \ \ \ $[1,6]: (02547)$\ \ \ \ \ \ $[2,3]: (4,7)$

$[2,4]: (13607)$\ \ \ \ \ \ $[2,6]: (45317)$\ \ \ \ \ \ $[2,7]: (3,6)$\ \ \ \ \ \ \ \ \ $[3,4]: (0,2)$ 

$[3,7]: (02546)$\ \ \ \ \ \ $[5,6]: (1,7)$\ \ \ \ \ \ \ \ \ $[5,7]: (02136)$\ \ \ \ \ \ $[6,7]: (2,5)$ 
\bigskip

$[0]: (21745)$\ \ \ \ \ \ $[1]: (02546)$\ \ \ \ \ \ $[2]: (13607)$\ \ \ \ \ \ $[3]: (21745)$\ \ \ \ \ \ 

$[4]: (13607)$\ \ \ \ \ \ $[5]: (13607)$\ \ \ \ \ \ $[6]: (21745)$\ \ \ \ \ \ $[7]: (02136)$\bigskip

Therefore, $F$ is not the prime graph of any solvable group of fifth-power-free order.
\end{proof}

\section{Groups of Cube-free order}
Groups of square-free order are automatically solvable, so their prime graphs can be characterized by directly applying Corollary \ref{nth power free simplified}, as we did in Corollary \ref{2nd power free}. In this section, we apply results from the previous two sections to characterize prime graphs of groups of cube-free order. Note that the solvable case has been taken care of by Corollary \ref{solvable cube free}. To obtain a characterization of non-solvable groups of cube-free order, we notice that these groups are of a particular form, namely they must be a direct product of some suitable $\PSL(2,p)$ with a group of odd and cube-free order. Therefore, we first consider the prime graphs of groups of odd and cube-free order.

\begin{lemmaN}\label{odd cube free}
$F$ is isomorphic to the prime graph of a group of odd and cube-free order if and only if its complement is bipartite.
\end{lemmaN}
\begin{proof}
By \cite[Theorem 1.1]{qiao2011finite}, a group of odd and cube-free order is metabelian. By Theorem \ref{first char.}, the prime graphs of such groups have a bipartite complement. \bigskip

Conversely, if a graph has bipartite complement, we can follow the same construction in Theorem \ref{first char.} while only using the odd primes, to obtain a group of square-free and odd order, which certainly is of cube-free and odd order.
\end{proof}

If $|G|$ is cube free order and $|G|$ not solvable, then $G=\PSL(2,p)\times M$, where $M$ is of odd order and $p$ is some suitable prime number.

\begin{lemmaN}\label{non-sol cube free}
$F$ is the prime graph of some non-solvable group of cube-free order if and only if $\overline{F}$ consists of two disconnected parts. One of them is bipartite and the other is obtained by taking some $\overline{\Gamma}(\PSL(2,p))$, $p\geq 5$ prime number, of cube-free order and deleting all edges in $\overline{\Gamma}(\PSL(2,p))$ connected to a subset $\cS$ of $$\left \{s\text{ odd prime number}: s\mid \dfrac{q(q+1)(q-1)}{2}, s^2\nmid\dfrac{q(q+1)(q-1)}{2} \right\}$$.
\end{lemmaN}
\begin{proof}
By \cite[Theorem 1.1]{qiao2011finite}, if $G$ is non-solvable of cube-free order, $G=\PSL(2,p)\times M$, where $|M|$ odd. Since $M\leq G$, $M$ is of cube-free order as well. By Lemma \ref{odd cube free}, the complement prime graph of $M$ is bipartite. Let $\cQ=\pi(\PSL(2,p))$ and $\cR=\pi(M)$.\bigskip

Note that if $a\in \PSL(2,p)$ and $b\in M$ with coprime order, since $a$ and $b$ commute, the order of $ab$ is the product of the order of $a$ and the order of $b$. \bigskip

If $\cQ\bigcap \cR=\emptyset$, $\overline{\Gamma}(G)$ has vertex set $\cQ\bigsqcup \cR$, where there is never an edge between $q\in \cQ$ and $r\in \cR$ by the above paragraph and Cauchy's theorem. So $\overline{\Gamma}(G)$ has two disconnected induced subgraphs of vertex set $\cQ$ and $\cR$., where $\Gamma(G)[\cQ]$ is the prime graph of $\PSL(2,p)$ and $\Gamma(G)[\cR]$ is the prime graph of $M$. This satisfies the statement where we take $S$ to be empty.\bigskip

If $\cQ\bigcap \cR=\cS\not=\emptyset$, any prime $s \in \cS$ divides both in $|\PSL(2,p)|$ and $|M|$. Since $G$ is of cube free order, $s^2$ cannot divide $|\PSL(2,p)|$. So $\cS$ is indeed a subset of $$\left\{s\text{ odd prime number}: s\mid \dfrac{q(q+1)(q-1)}{2}, s^2\nmid\dfrac{q(q+1)(q-1)}{2}\right\}$$ Let $\cQ=\cQ'\bigsqcup \cS$ and $\cR=\cR'\bigsqcup \cS$. Then the prime graph of $G$ has vertex set $\cP'\bigsqcup \cQ'\bigsqcup \cR$. Since elements in $\PSL(2,p)$ and $M$ commute with each other, all vertices in $\cS$ are not adjacent to any other vertex in $\overline{\Gamma}(G)$. Also, any vertex in $\cQ'$ is not adjacent to any vertex in $\cR'$ in $\overline{\Gamma}(G)$. Now we have that the subgraph $\Gamma[Q'\bigsqcup S]$ is obtained by adding an edge between any point in $Q'$ and any point in $S$ to the prime graph of $\PSL(2,p)$. And the subgraph $\Gamma[R']$ is the prime graph of the Hall $R'$-subgroup of $M$, whose order is odd and cube free. (M solvable, so Hall-subgroups always exist.) \bigskip

Conversely, let $Q=Q'\bigsqcup S$, where $Q$ is the set of prime divisors of $\PSL(2,p)$, and $S$ is the subset of it such that all edges connected to points in $S$ in the complement graph is removed. Let $R$ be the remaining vertices.\bigskip

Note that the subgraph $R$ has bipartite complement, by our previous construction, there exists $M'$, a group of square free order whose prime graph is our graph restricted to $R$. Note that in our previous construction, we have infinite suitable primes to choose from, so we can always avoid divisors of $\PSL(2,p)$\bigskip

Let $M$ be the direct sum of $M'$ and $\Z_s$ for each $s\in S$. Then $\PSL(2,p)\times M$ has the desired prime graph. Also, $\PSL(2,p)\times M$ still has cube free order, because the prime divisors of $M$ either do not appear in $|\PSL(2,p)|$, or appear in both $M$ and $\PSL(2,p)$ exactly once. 
\end{proof}
Since odd order groups are always solvable, the class of all groups of cube-free order can be divided into groups of cube-free and odd order, solvable groups of cube-free and even order, and non-solvable groups of cube-free order.
\begin{theoremN}\label{cube free order}
Below, we give a stronger result than the characterization of prime graphs of groups of cube free order.
\begin{enumerate}
    \item $F$ is isomorphic to the prime graph of some group of cube-free and odd order if and only if $\overline{F}$ is bipartite.\\
    \item $F$ is isomorphic to the prime graph of some solvable group of cube-free and even order if and only if $\overline{F}$ is triangle free and is $2$-colorable after removing one vertex.\\
    \item $F$ is isomorphic to the prime graph of some non-solvable group of cube-free order if and only if $\overline{F}$ consists of two disconnected parts. One of them is bipartite and the other is obtained by taking the complement of a prime graph of some $\PSL(2,p)$ of cube-free order and deleting all edges connected to a subset $S$ of $$\left \{s\text{ odd prime number}: s\mid \dfrac{q(q+1)(q-1)}{2}, s^2\nmid\dfrac{q(q+1)(q-1)}{2} \right\}$$.
\end{enumerate}

\end{theoremN}
\begin{proof}
By Lemma \ref{odd cube free}, Lemma \ref{non-sol cube free}, and Lemma \ref{solvable cube free}.
\end{proof}

\begin{lemmaN}
$F$ is isomorphic to the prime graph of some $\PSL(2,q)$ of cube-free order, where $q\geq 5$ prime, if and only if $F$ has three connected components, which are an isolated point, an $m$-clique, and an $n$-clique, where $m$ and $n$ satisfy the following number theory condition:

There exists a prime number $q\geq 5$ such that
\begin{enumerate}
    \item $q\equiv 3,5 \mod  8$
    \item Both $q+1$ and $q-1$ are cube-free. 
    \item One of $\dfrac{q+1}{2}$ and $\dfrac{q-1}{2}$ has $m$ distinct prime divisors, and the other one has $n$ distinct prime divisors.
\end{enumerate} 
\end{lemmaN}
\begin{proof}
By a result of Dickson (see \cite[II, Hauptsatz 8.27]{huppert2013endliche}) we know that the set of element orders of $\PSL(2,q)$, where $q\geq 5$ prime, is the set of prime divisors of $q$, $\frac{q+1}{2}$, and $\frac{q-1}{2}$. Therefore, if $\frac{q+1}{2}$ has $m$ prime divisors and $\frac{q-1}{2}$ has $n$ prime divisors, the prime graph of $\PSL(2,q)$ consists of an isolated point, an $m$-clique, and an $n$-clique, mutually disconnected with each other.
\end{proof}

\begin{remark}
This completes Theorem \ref{cube free order} in that it tells us what exactly are the graphs of $\PSL(2,q)$ of cube free order.
\end{remark}

\section{Groups of Composition Factors Cyclic or $A_5$}
In this section, we study prime graphs of groups whose composition factors are either cyclic or $A_5$. We call these groups pseudo-solvable, and we call nonsolvable pseudo-solvable groups strictly pseudo-solvable. In the literature, a group is often said to be almost solvable if the quotient by its solvable radical is trivial or is isomorphic to $A_5$. Our notion of pseudo-solvability is its generalization.\bigskip

Notice that a group is solvable if and only if its composition factors are all cyclic. Now, by allowing a non-abelian simple group $A_5$ in the composition series, we are giving up solvability on the groups that we are working on. As a consequence, we do not have the general setting and tools established in Section 2 to study prime graphs of these groups. In particular, we do not have the following.
\begin{enumerate}
    \item Lucido's Three Primes Lemma, which would give us that the complement is triangle free.
    \item Hall subgroups exist, which would allow us to zoom in onto induced subgraphs of the prime graph.
    \item If the prime graph is disconnected, the group is Frobenius or 2-Frobenius, which would allow us to direct the edges in the complement of the prime graph and study the directed paths.
\end{enumerate}

To deal with this difficulty, we will seek alternatives to (2) and (3). Hall subgroups are easy to work with because their prime graphs are induced subgraphs of the prime graph of the whole group. In general, if $K$ is a subgroup of $ G $, $\Gamma(K)$ is a subgraph of $\Gamma(G)$, though not necessarily an induced subgraph. In other words, $\overline{\Gamma}(K)$ will have fewer vertices, but more edges among the vertices. The same holds if $K$ is a quotient of $G$. This is not ideal as in (2), but we have that all conditions that are closed under removing edges on $\overline{\Gamma}(K)$ will hold on the induced subgraph of $\overline{\Gamma}(G)$ by $V(K)$, which allows us to make the reduction to subgroups. To find a substitution for (3), we attempt to direct an edge in $\overline{\Gamma}(G)$ by viewing it as an edge in $\overline{\Gamma}(K)$ for some certain solvable subgroup $K\leq G$. Before showing the details of this orientation, we start with several technical lemmas. \bigskip

\begin{lemmaN}\label{Semidirect second lemma}
Let $N$ normal in $G$ and define $\phi: G\to \Aut(N)$ by sending $g$ to $\phi_g:n\mapsto gng^{-1}$. If $\phi_h\in \Inn(N)$ for some $h\not\in N$ of order $q$ in $G$, then there exists an element in $G$ of order $pq$ for any $p\mid |N|$ such that $p \not= q$. In particular, if $q\nmid |Aut(N)|$, $\phi_h=1$ and therefore $\phi_h \in \Inn(N)$.
\end{lemmaN}
\begin{proof}
If $q\mid |Z(N)|$, we are already done, because there exists an element in $N$ of order $pq$ for any $p\mid |N|$. Now we assume otherwise. Since $\phi_h\in \Inn(N)$, there exists $a\in N$ such that $\phi_h=\phi_a$. Since $h^q=1\in G$, $(\phi_a)^q=(\phi_h)^q=1$. So $a^q=z\in Z(N)$.\bigskip

Suppose $z$ has order $k$, then $q\nmid k$ since $q\nmid |Z(N)|$ by assumption. There exists integer $l$ such that $q\mid lk-1$. Let $z'=z^{\frac{lk-1}{q}}\in Z(N)$ and $a'=az'$. Then $ana^{-1}=(az')n(az')^{-1}=a'n(a')^{-1}$. Furthermore, $(a')^q=(az')^{q}=a^q(z')^q=z (z')^q=zz^{lk-1}=z^{lk}=1$. Therefore, we can assume without loss of generality that there exists $a\in N$ such that $a^q=1$ and $\phi_h=\phi_a$. \bigskip

For any $p\mid |N|$, let $b\in N$ be an element of order $p$. $\phi_a(b^{k}a^{-k})=\phi_h(b^{k}a^{-k})$ gives us that $a^{-1}hb^{k}a^{-k}=b^k a^{-k-1}h$. Prove by induction that $(ba^{-1}h)^k=b^k a^{-k}h^k$. When $k=1$, this holds trivially. Assuming $(ba^{-1}h)^k=(b^k a^{-k}h^k)$, we have that $(ba^{-1}h)^{k+1}=b(a^{-1}hb^k a^{-k}) h^k=b(b^k a^{-k-1}h)h^k=b^{k+1}a^{-k-1}h^{k+1}$. So indeed, we have that $(ba^{-1}h)^k=b^k a^{-k}h^k$. We will show that $ba^{-1}h$ has order $pq$.\bigskip

Note that the order of $hN$ in $G/N$ divides $q$, the order of $h$ in $G$. Since $h\in N$, $hN$ has order $q$ in $G/N$. Now, if $1=(ba^{-1}h)^k=b^ka^{-k}h^k$, we have that $h^k\in N$ because $a$ and $b$ are both in $N$. Then $q\mid k$ because $q$ is the order of $hN$ in $G/N$. This gives us that $a^{-k}=1$, because $a^q=1$ by our choice of $a$. Then, $1=b^ka^{-k}h^k=b^k$, so $p\mid k$. Therefore, $pq\mid k$ if $1=(ba^{-1} h)^k$, and the converse is obviously true. So $ba^{-1}h$ has order $pq$.  \end{proof}

\begin{lemmaN}
Let $G$ be strictly pseudo-solvable, then $G$ contains a subgroup $K\simeq N.(A_5\times H)$, where $N$ solvable and $H$ solvable of order coprime to $30$. Furthermore, $|K|$ and $|G|$ share the same set of prime divisors.
\end{lemmaN}
\begin{proof}
First, build a normal series of $G$ by factoring out a minimal normal subgroup of the quotient at each step. Since minimal subgroups of a finite group must be a direct product of the same simple group, we have that each factor in the normal series is either elementary abelian or $A_5^k$, a direct product of $k$ copies of $A_5$. Let $\pi=\{2,3,5\}$. Since the composition factors of $G$ are either cyclic or $A_5$, $G$ is $\pi$-separable.\bigskip

Suppose that we have in the normal series, the lowest factor is $A_5^k$. Since $G$ is $\pi$-separable, $G$ contains a Hall $\pi$-subgroup $H$. Note that $|H|$ is coprime to $30$ and is divisible by any other prime that divides $|G|$. Since $A_5^k\nor G$, $H$ acts on $A_5^k$ by conjugation. Notice that each nontrivial element in $H$ has order coprime to 30, so its image under the group homomorphism $\phi: H\to \Aut(A_5^n)=S_5^n\times S_n$ must be in $\{0\}^n\times S_n$, where the equality is from \cite[Theorem 3.1]{Bidwell2008}. In other words, each element in $H$ acts by permuting the $n$ copies of $A_5$. Consider $L=\{(a,a,\cdots, a):a\in A_5\}$ a subgroup of $A_5^n$. $L$ is $H$-invariant and isomorphic to $A_5$. Similarly, the image of any element in $H$ under $\phi':H \to \Aut(L)=\Aut(A_5)=S_5$ must be trivial by considering its order, which indicates that $H$ acts trivially on $L$. Also, $H$ and $L$ has trivial intersection. So $K$ has a subgroup $L\times H\simeq A_5\times H$.\bigskip

In general, there would be a collection of elementary abelian groups below the first $A_5^k$ in the normal series. Let $N$ be the normal subgroup of $G$ just below the first $A_5^k$ in the normal series. Apply the previous argument to $G/N$ to get a subgroup $K'\simeq A_5\times H$ of $G/N$, where $H$ is the Hall $\pi$-subgroup of $G/N$. So the preimage of $K'$ under the projection $G\to G/N$ is a subgroup of $G$ isomorphic to $N.(A_5\times H)$. 
\end{proof}
\begin{corollaryN}\label{first reduction of C groups}
Let $G$ be strictly pseudo-solvable, then $\overline{\Gamma}(G)$ is obtained by removing edges from $\overline{\Gamma}(K)$ for some $K\simeq N.(A_5\times H)$, where $N$ solvable and $H$ solvable of order coprime to $30$. 
\end{corollaryN}
Thereby, we have reduced the problem of studying pseudo-solvable groups to groups of the above form.
\begin{lemmaN} \label{Hall respects extension}
Let $H$ be a Hall $\pi$-subgroup of $G$ and $N$ normal in G. Then in the normal series $(H\bigcap N).(H/H\bigcap N)\simeq(H\bigcap N).(HN/N)$ of $H$, the factor $H\bigcap N$ is a Hall $\pi$-subgroup of $N$, and the factor $HN/N$ is a Hall $\pi$-subgroup of $G/N$.
\end{lemmaN}
\begin{proof}
Naively, we have $|H\bigcap N|$ divides $|H|$. Since $[N:H\bigcap N]=[HN:H]$ divides $[G:H]$, we have that $\gcd{|H\bigcap N|,[N:H\bigcap N]}$ divides $\gcd{|H|,[G:H]}=1$. So $H\bigcap N$ is a Hall $\pi$-subgroup of $N$.\bigskip

On the other hand, $[HN:N]=[H:H\bigcap N]$ divides $|H|$ and $[(G/N):(HN/N)]=[G:HN]$ divides $[G:H]$. So $\gcd{[HN:N],[(G/N):(HN/N)]}$ divides $\gcd{|H|,[G:H]}=1$. So $HN/N$ is a Hall $\pi$-subgroup of $G/N$.
\end{proof}
\begin{remark}
By iterating the above lemma, we can construct a normal series for a Hall subgroup $H$ of $G$ based on a normal series of $G$. And each factor in the normal series of $H$ is isomorphic to a Hall subgroup of the corresponding factor in the normal series of $G$.
\end{remark}
\begin{lemmaN}\label{no q on top}
If $G$ is Frobenius of type $(p,q)$ or $2$-Frobenius of type $(p,q,p)$, then there does not exist a normal subgroup $N$ of $G$ such that $G/N$ is a nontrivial $q$-group.
\end{lemmaN}
\begin{proof}
Let $S=O_p(G)$ the Sylow $p$-subgroup of $F(G)$, so $S=1$ when $G$ is Frobenius of type $(p,q)$ and $S$ is a $p$-group when $G$ is $2$-Frobenius of type $(p,q,p)$. Let $H=G/S$, then $H$ is a Frobenius group of type $(p,q)$.\bigskip

Now assume that there exists a normal subgroup $N$ of $G$ such that $G/N$ is a nontrivial $q$-group. $N$ is clearly not a subgroup of $S$ because the quotient by any subgroup of $S$ will have order divisible by $p$. So $M=NS/S >1$, and $M$ is a normal subgroup of $H$ such that $H/M$ is a nontrivial $q$-group. By now, we have reduced the problem to proving the statement for $H$ and $M$.\bigskip

Next, take $M\nor H$ such that $H/M$ is the largest quotient of $H$ that is a $q$-group. Let $L\nor H$ such that $L$ is a subgroup of $M$ and $M/L$ is a minimal normal subgroup of $H/L$. Then $H/L$ is an
elementary abelian $p$-group.\bigskip

Now let $x$ be a nonidentity element of the (unique, normal) Sylow $q$-subgroup $Q$ of $H$ such that $x$ is not an element of $M$. Also, let $y$ be an element of $M$ such that $y$ is not contained in $L$. Let $z=[x,y]=x^{-1}y^{-1}xy$ be the commutator of $x$ and $y$. We have the following.
\begin{enumerate}
    \item $x$ is an element of $Q$ on which $y$ acts Frobeniusly, so $z=(x^{-1}) x^y$ is in $Q\setminus\{1\}$, and thus is a nontrivial element of $q$-power order.
    \item $y$ is in $M\setminus L$, and as seen in (1), $z$ is not the identity element, and also $z=(y^{-1})^x y$ is an element of $M\setminus L$. Since $M/L$ is a $p$-group, this shows that $z$ has order divisible by $p$.
\end{enumerate}

(1) and (2) contradicts with each other.

\end{proof}

\begin{theoremN}\label{weak result on C-groups}
Let $G$ be pseudo-solvable, then $\overline{\Gamma}(G)\setminus\{3\text{-}5\}$ is $3$-colorable and triangle free.
\end{theoremN}
\begin{proof}
Note that the condition of being $3$-colorable and triangle free is closed under removing edges. If $G$ is itself solvable, $\overline{\Gamma}(G)$ is already $3$-colorable and triangle free, then so is $\overline{\Gamma}(G)$ with possibly one edge removed. In the case where $G$ is not solvable, it suffices to prove the statement for $K=N.(A_5\times H)$ by Corollary \ref{first reduction of C groups}. Furthermore, if $3$ divides $|N|$, consider the subgroup $K_1=N.(D_{10}\times H)$ of $K$. $K_1$ has the same set of prime divisors as $K$, so $\overline{\Gamma}(K)$ is obtained by removing edges from $\overline{\Gamma}(K_1)$, which is triangle free and $3$-colorable because $K_1$ is solvable. Therefore, $\overline{\Gamma}(K)$ is already triangle free and $3$-colorable, then so is $\overline{\Gamma}(K)$ possibly with one edge removed. Similarly, by taking $K_2=N.(A_4\times H)$, we know that if $5$ divides $|N|$, the statement is also true. Therefore, we may assume $K=N.(A_5\times H)$, where $N$ is solvable of order coprime to $15$ and $H$ is solvable of order coprime to $30$.\\

Still, we consider two solvable subgroups $K_1=N.(D_{10}\times H)$ and $K_2=N.(A_4\times H)$ of $K$. Since $|N|$ is coprime to $15$ and $|H|$ is coprime to $30$, we have that the prime divisors of $K_1$ are all prime divisors of $K$ except $3$. Therefore, $\overline{\Gamma}(K)[\{3\}']$ is obtained by removing edges from $\overline{\Gamma}(K_1)$. Similarly, $\overline{\Gamma}(K)[\{5\}']$ is obtained by removing edges from $\overline{\Gamma}(K_2)$. In other words, if an edge $pq$ does not include the vertex $3$ (or resp., $5$), then it is contained in $\overline{\Gamma}(K_1)$ (or resp., $\overline{\Gamma}(K_2)$).  We attempt to assign an orientation to each edge in $\overline{\Gamma}(K)$ that is not $3\text{-}5$ by taking its orientation in $\overharpo{\Gamma}(K_1)$ and/or $\overharpo{\Gamma}(K_2)$. It suffices to show that if $p\text{-}q$ is an edge in $\overline{\Gamma}(K)\setminus\{3\text{-}5\}$, then its orientation in $\overharpo{\Gamma}(K_1)$ and $\overharpo{\Gamma}(K_2)$ coincide. We have the following cases.
\begin{enumerate}
    \item If $p,q\not \in \{2,3,5\}$, the Hall $\{p,q\}$-subgroups of $K_1$ and $K_2$ are both isomorphic to the Hall $\{p,q\}$-subgroups of $N.H$. So the directions of their Frobenius actions naturally coincide
    \item If $p\in \{3,5\}$ and $q\not \in \{3,5\}$, $p\text{-}q$ appears in either $\overline{\Gamma}(K_1)$ or $\overline{\Gamma}(K_2)$, depending on $p$ being $3$ or $5$. So there is no ambiguity in the orientation.
    \item If $p=2$ and $q\not\in \{2,3,5\}$, consider Hall $\{2,q\}$-subgroups $H_1$ of $K_1$ and $H_2$ of $K_2$. By Lemma \ref{Hall respects extension}, $H_1\simeq A.(V_4\times B)$ and $H_2\simeq A.(\Z_2\times B)$, where $A$ is the Hall $\{2,q\}$-subgroup of $N$, $B$ is the Sylow $q$-subgroup of $H$, $V_4$ is a Sylow $2$-subgroup of $A_4$ isomorphic to the Klein $4$ group. Since $A.B$ is a normal subgroup such that the quotient in both $H_1$ and $H_2$ is a $2$-group, the orientation must both be from $2$ to $q$ by Lemma \ref{no q on top}.  

\end{enumerate}
Thereby, we have proved that there is a well-defined orientation for any edge $p\text{-}q$ in $\overline{\Gamma}(K)\setminus \{3\text{-}5\}$ such that it coincides with the orientation in $\overharpo{\Gamma}(K_1)$ and $\overharpo{\Gamma}(K_2)$. Define $\overharpo{\Gamma}(K)$ to be this orientation of $\overline{\Gamma}(K)\setminus \{3\text{-}5\}$. Call $\overharpo{\Gamma}(K)$ the Frobenius Diagraph of $K$.\bigskip

Furthermore, if $p=3$ and $q\not\in \{2,3,5\}$ is an edge in $\overharpo{\Gamma}(K)$, by Lemma \ref{Hall respects extension}, take the Hall $\{3,q\}$ subgroup of $K_2$ isomorphic to $A.(\mathbb{Z}_3)$. Note that here, $|H|$ must not be divisible by $q$, because otherwise there exists an element of order $3q$ in the quotient, and thus in $K_2$, which is a contradiction. Therefore, $A$ is a normal subgroup such that the quotient is a $3$-group, then the orientation must be $3\to q$ by Lemma \ref{no q on top}. Similarly, if $p=5$ and $q\not\in \{2,3,5\}$ is an edge in $\overharpo{\Gamma}(K)$, the orientation must be $5\to q$ by taking the Hall $\{5,q\}$ subgroup of $K_1$ isomorphic to $A.\Z_5$.\bigskip

Similarly, take the Hall $\{2,3\}$-subgroup of $K_2$ isomorphic to $I.A_4=I.(V_4\rtimes  \Z_3)$. We have that if there is an edge $2\text{-}3$, it must be $3\to 2$. Take the Hall $\{2,5\}$-subgroup of $K_1$ isomorphic to $J.D_{10}=J.(\Z_5\rtimes \Z_2)$. We have that if there is an edge $2\text{-}3$, it must be $2\to 5$.\bigskip

In summary, in $\overharpo{\Gamma}(K)$, we must have $p\to q$ for $p\in \{2,3,5\}$ and $q\not\in \{2,3,5\}$, $3\to 2$, and $2\to 5$ if these edges exist. Also, $\overharpo{\Gamma}(K)$ has no directed $3$-path with either $3$ or $5$ removed.\bigskip

Therefore, any directed $3$-path in $\overharpo{\Gamma}(K)$ must contain both $3$ and $5$. Since $3$ and $5$ cannot be connected in $\overharpo{\Gamma}(K)$ and we cannot have any edge into $3$, $3$ must be the start of the $3$-path. Since the only edge into $5$ must be $2\to 5$ and the only edge into $2$ must be $3\to 2$, any $3$-path in $\overharpo{\Gamma}(K)$ must be of the form $3\to 2\to 5 \to p$ for some $p\not \in \{2,3,5\}$. Define a new orientation of $\overline{\Gamma}(K)\setminus \{3\text{-}5\}$ by reverse the direction from $2\to 5$ to $5\to 2$. Since we only reversed the edge $2\text{-}5$ in the new orientation, any $3$-path must contain the edge $2\text{-}5$ or it would also be present in the original orientation, which will not be of the form $3\to 2\to 5 \to p$, which is a contradiction. Since there is no edge coming into $5$, $5$ must be the start of the $3$-path. So the this $3$-path must be of the form $5\to 2\to r\to s$. But this implies that $3\to 2\to r\to s$ is a $3$-path in the original orientation. Contradiction. Therefore, the new orientation has no directed $3$-path. So $\overline{\Gamma}(K)$ is $3$-colorable with $3\text{-}5$ removed.\bigskip

Notice that $\overline{\Gamma}(K)$ is triangle free with either $3$ or $5$ removed. So if there exists any triangle in $\overline{\Gamma}(K)$, it must contain both $3$ and $5$. So $\overline{\Gamma}(K)$ with $3\text{-}5$ removed is triangle free. \bigskip

Thereby, we have proved that $\overline{\Gamma}(G)$ with the edge $3\text{-}5$ removed, if any, is triangle free and $3$-colorable.
\end{proof}

\begin{remark}
It is unclear at this point if the condition can be strengthened. One particular example of consideration is the Groetzsch graph, which satisfies the assumptions of Theorem \ref{weak result on C-groups} but we do not know if it can be the prime graph of a pseudo-solvable group. Note that the Groetzsch graph happens to be triangle free but not $3$-colorable. This draws connection to Maslova's Conjecture that a prime graph cannot be triangle free if it is not $3$-colorable (see \cite[Problem 19.52]{khukhro2014unsolved}).
\end{remark}

We proceed by proving some other properties for prime graphs of pseudo-solvable groups.

\begin{lemmaN}
Let $G$ be pseudo-solvable. If $\{2,3,5\}$ forms a triangle in $\overline{\Gamma}(G)$, then there is exactly one copy of $A_5$ in the composition series. Furthermore, take any normal series of $G$, the quotient by everything up to the first $A_5$ must have order coprime to $30$.
\end{lemmaN}
\begin{proof}
Take the normal series of $G$ where each factor is a minimal normal subgroup of the quotient. We can quotient out everything below the first $A_5$, so we can assume that we have $A_5^k$ normal in $G$. Since $\overline{\Gamma}(G)$ is not $3$-colorable and triangle free, $\{2,3,5\}$ must form a triangle in $\overline{\Gamma}(G)$. If $k\geq 2$, $A_5^2$ contains an element of order $15$. So $3\text{-}5$ is not an edge in $\overline{\Gamma}(G)$. Contradiction. Therefore, we can assume that $A_5$ normal in $G$. Let $G= A_5 . K$. It suffices to show that $|K|$ is coprime to $30$.\bigskip

Since $G$ acts on $A_5$ by conjugation, we have the homomorphism $\phi: G\to \Aut(A_5)=S_5$. The image of $\phi$ at least contains $\Inn(A_5)=A_5$. If the image of $\phi$ equals $\Inn(A_5)$, we conclude that there is an element of order 15 so we are done by Lemma \ref{Semidirect second lemma}.\bigskip

Otherwise, the image of $\phi$ must equal $S_5$. Let $C$ be the centralizer of $A_5$ in $G$, which is the kernel of $\phi$. So $G/C=S_5$, which indicates that $|C|=|G|/120=|K|/2$. For any $p\not =2$, if $p$ divides $|K|$, $p$ divides $|C|$. Take any element $c\in C$ of order $p$ and let $e$ be an involution in $A_5$. We have that $ce$ has order $2p$, since they commute. So there must be an edge $2\text{-}p$ in $\overline{\Gamma}(G)$. This contradicts with $\{2,3,5\}$ forming a triangle in $\overline{\Gamma}(G)$. Therefore, it suffices to show that $|K|$ is odd. (Similarly, we could also take $e\in A_5$ to have order $3$ or $5$. This would give us that there cannot be edges $2\text{-}p$, $3\text{-}p$, and $5\text{-}p$ in $\overline{\Gamma}(G)$.)\bigskip

If $|K|$ is even, then $K$ has a subgroup isomorphic to $\Z_2$. Take the preimage of this subgroup under the projection $G\to G/A_5$ to obtain a subgroup of $G$ isomorphic to $A_5. \Z_2$. It is well-known that $A_5. \Z_2$ must either be $S_5$ or $A_5\times \Z_2$. Since both contains an element of order $6$, $\{2,3,5\}$ cannot form a triangle in $\overline{\Gamma}(G)$. Contradiction.
\end{proof}

\begin{theoremN}\label{second lemma on C-groups}
Let $G$ be pseudo-solvable. One of the following holds.
\begin{enumerate}
    \item $\{2,3,5\}$ do not form a triangle in $\overline{\Gamma}(G)$.
    \item $2\text{-}p$ is not an edge in $\overline{\Gamma}(G)$ for any $p\not\in \{2,3,5\}$.
\end{enumerate}
\end{theoremN}
\begin{proof}
Assume that there is an edge $2\text{-}p$ in $\overline{\Gamma}(G)$ for $p\not\in {2,3,5}$ and that $\{2,3,5\}$ forms a triangle. By the above lemma, we have a normal series of $G$ of the form $N.A_5.H$, where $N$ and $H$ solvable and $|H|$ coprime to $30$.\bigskip

By the comment at the end of the penultimate paragraph of the proof of the above lemma, we also have that $p$ does not divide $|H|$ if $2\text{-}p$ is an edge in the complement prime graph. Therefore, $p$ divides $|N|$. So $\Gamma(N.A_5)$ is a subgraph of $\Gamma(G)$ containing vertices $\{2,3,5,p\}$. Therefore, $2\text{-}p$ is an edge and $\{2,3,5\}$ forms a triangle in $\Gamma(N.A_5)$. So far, we have reduced the problem to the case where $H=1$.\bigskip

Note that $N.A_5$, though not solvable, does have a Hall $\{2,p\}$-subgroup because $N.A_4$ has a Hall $\{2,p\}$ subgroup for $p\not\in \{2,3,5\}$ and $[N.A_5:N.A_4]=5$ is coprime to $2p$. By Lemma \ref{Hall respects extension}, we can form a normal series of a Hall $\{2,p\}$ subgroup $H$ of $N.A_5$ such that each of its factor is isomorphic to a Hall $\{2,p\}$-subgroup of each factor in the normal series of $N.A_5$. Therefore, the factor on the top of this normal series of $H$ would be isomorphic to $V_4$, a Hall $\{2,p\}$ subgroup of $A_5$, and each factor below is either an elementary abelian $2$-group or an elementary abelian $p$-group. By \cite[Theorem A]{WILLIAMS1981487}, $H$ is Frobenius or $2$-Frobenius. By Lemma \ref{no q on top}, since we have $V_4$ on the top of a normal series, $H$ must be Frobenius of type $(2,p)$ or $2$-Frobenius of type $(2,p,2)$. Either way, the quotient of $H$ by the Sylow $2$-subgroup of $F(H)$ is Frobenius of type $(2,p)$ and has $V_4$ on the top of its normal series.\bigskip

By \cite[Theorem 6.3]{isaacs2008finite}, a Frobenius kernel of even order contains a unique element of order $2$. Therefore, a minimal normal subgroup of the Frobenius complement must be $\Z_2$ instead of some elementary abelian $2$-group of higher order. Therefore, $\Z_2$ appears as a factor in the normal series of $H$. Since we constructed this normal series of $H$ such that each of its factor is isomorphic to a Hall $\{2,p\}$-subgroup of each factor in the normal series of $N.A_5$ and all factors below $A_5$ are elementary abelian groups, we have that there is a $\Z_2$ as a factor in the normal series of $N.A_5$. Therefore, $N.A_5$ has a quotient $T\simeq \Z_2.S.A_5$. Consider the group action of $T$ on $\Z_2$ given by $\phi: T \to \Aut(\Z_2)\simeq \Z_2$. Now we apply Lemma \ref{Semidirect second lemma}. Since there is an element in $T \setminus \Z_2$ of order $3$ and $3$ doesn’t divide $|\Aut(\Z_2)|=2$, there is an element in $T$ of order $6$. Similarly, $T$ contains an element of order $10$. So $H$, and furthermore, $G$, contains an element of order $6$ and an element of order $10$. This contradicts with our assumption that $\{2,3,5\}$ form a triangle in $\overline{\Gamma}(G)$.
\end{proof}

\section{Dual Prime Graphs}
We start by introducing a more general set-up.
\begin{definition}
Let $X$ be a set of integers.
\begin{enumerate}
    \item The vertex set of $\Delta(X)$, the prime vertex graph of $X$, is the set of all prime numbers dividing some element of $X$. $pq$ is an edge if and only if $pq$ divides some element of $X$.
    \item The vertex set of $\Gamma(X)$, the common divisor graph, is $X^*=X\setminus \{1\}$. $ab$ is an edge if and only if $a$ and $b$ have nontrivial gcd. We remove $1$ from the graph because it is always an isolated point.
\end{enumerate}
\end{definition}
We can see that $\Gamma(X)$ and $\Delta(X)$ are in some sense dual to each other. For a finite group $G$, there are multiple ways of associating $G$ with a set of integers $X$. We can take $X$ as $\eo(G)$ the set of all element orders in $G$, as $\cd(G)$ the set of all character degrees of $G$, or as $\cs(G)$ the set of all conjugacy class sizes of $G$. Notice that when we take $X=\eo(G)$, $\Delta(\eo(G))$ is the prime graph of $G$ defined earlier, and $\Gamma(\eo(G))$ is its dual. The notation is a little unfortunate, since $\Gamma(G)$ was used to denote the prime graph of $G$, but these are the standard notations in the literature. We will try to make this section less confusing by only using $\Delta(\eo(G))$ (or simply $\Delta$) and $\Gamma(\eo(G))$ (or simply $\Gamma$) to denote the prime graph and the dual prime graph of $G$. \bigskip

$\Delta (\eo(G)), \Delta(\cd(G)), \Delta(\cs(G)), \Gamma(\cd(G))$, and $\Gamma(\cs(G))$ have all been studied, but seemingly not $\Gamma(\eo(G))$. Although it is unclear whether there are connections among $\eo(G)$, $\cd(G)$, and $\cs(G)$, their graphs share similar-looking results. This raised our interests in studying $\Gamma(\eo(G))$. We devote this section to constructing an algorithm that given a graph $F$ isomorphic to the dual prime graph $\Gamma(\eo(G))$ of some group $G$, recovers $\Delta(\eo(G'))$ for any group $G'$ such that $\Gamma(\eo(G'))$ is isomorphic to $F$. Furthermore, the algorithm has a high frequency to reject a graph that is not isomorphic to the dual prime graph of any group. In fact, we give this algorithm in a more general setting of $\Delta(X)$ and $\Gamma(X)$, where $X$ is a set of integers closed under taking divisors.\bigskip

In this section, $\pi(m)$ denotes the set of distinct prime divisors of a natural number $m$ and $\Pi(m)$ denotes the product of these primes in $\pi(m)$. $N(v)$ denotes the neighborhood of $v$, which consists of $v$ and all vertices adjacent to $v$.

\begin{lemmaN}\label{Contract}
A subgraph $S$ of $\Gamma(X)$ is a complete graph and $N(m)$ is identical for all $m\in S$ if and only if $\pi(m)$ is identical for all $m\in S$. Furthermore, if $S$ is a maximal such subgraph, then  $S=\{m\in X : \pi(m)=\pi(n) \}$ for each $n\in S$.
\end{lemmaN}
\begin{proof}
($\Leftarrow$) For any $m,n\in S$, $m$ and $n$ share the same set of prime divisors. Therefore, $m$ and $n$ are adjacent in $\Gamma(X)$ and $N(m)=N(n)$. \bigskip

($\Rightarrow$) Conversely, assume that we have a subgraph $S$ of $\Gamma(X)$ such that $S$ is a complete graph and $N(m)$ is the same for all $m\in S$. Let $L=\gcd{m:m \in S}$. If there exists $m\in S$ such that $m=Lm'$ with $\pi(m')\not\subset \pi(L)$, take a prime $p\in \pi(m')- \pi(L)$. $p\not\in L$ implies that $\gcd{p,L}=1$. There exists $n\in S$ such that $Lp\nmid n$ because $L$ is the gcd. Now that $L\mid n$, $\gcd{p,L}=1$, and $p\nmid n$, so $p\nmid n$. Then $p\in N(m)$ but $p\not\in N(n)$. This is a contradiction to $N(m)=N(n)$. Therefore, for all $m
\in S$, $\pi(m')\subset \pi(L) $, which implies $\pi(m)=\pi(L)$. So $\pi(m)$ is the same for all $m\in S$.\bigskip

For the second part, let $S$ be a maximal such subgraph. $S\subset\{m\in \eo(G) : \pi(m)=\pi(n) \}$ for each $n\in S$ because $\pi(m)$ is the same for all $m\in S$. Also, $\{m\in \eo(G) : \pi(m)=\pi(n) \}$ is itself such a subgraph for each $n\in S$. So $S=\{m\in \eo(G) : \pi(m)=\pi(n) \}$ for each $n\in S$.
\end{proof}
\begin{remark}
By the above lemma, we can identify all the non-square free order vertices in $\Gamma(X)$ and group them with their square-free order divisors. We replace the grouped vertices with a single vertex labeled by the square-free order divisor and call the resulting graph the contracted common divisor graph. Alternatively, we can view the contracted common divisor graph as the induced subgraph of the common divisor graph on the set of all square-free vertices.\bigskip

Now, we try to recover the prime vertex graph from the contracted common divisor graph.
\end{remark}
\begin{lemmaN}
Since all vertices are square-free in a contracted common divisor graph, no two vertices share the same neighborhood. Given two vertices $v$ and $u$, $v\mid u$ if and only if $N(v)\subset N(u)$. 
\end{lemmaN}
\begin{proof}
($\Rightarrow$) If $v\mid u$, $\gcd{v,w}\mid \gcd{u,w}$ for any $w$. So $N(v)\subset N(u)$.\bigskip

($\Leftarrow$) If $v\nmid u$, let $p$ prime such that $p\mid v$ but $p\nmid u$. Then $p\in N(v) $ but $p\not \in N(u)$. Contradict to $N(v)\subset N(u)$.
\end{proof}
\begin{corollaryN}\label{posetisom}
In a contracted common divisor graph, the poset of the vertices ordered by inclusion of neighborhood is isomorphic to the poset of their values ordered by divisibility. Thus, the poset is graded by the number of prime divisors.
\end{corollaryN}
\begin{theoremN}\label{the algorithm}
There exists an algorithm that does the following. Given a graph that is isomorphic to the common divisor graph $\Gamma(X)$ of some set $X$ that is closed under taking divisors, it recovers a graph isomorphic to the prime vertex graph $\Delta(X')$ of $X'$ for any $\Gamma(X')=\Gamma(X)$ and $X'$ closed under taking divisors. Furthermore, if the input graph is not isomorphic to the common divisor graph of any $X$ that is closed under taking divisors, it is very likely that our algorithm rejects it.
\end{theoremN}
\begin{proof}
We will give the algorithm as follows and explains why it works as desired.\bigskip
\begin{enumerate}
    \item By Lemma \ref{Contract}, identify the non-square-free terms and form the contracted common divisor graph.
    
    \item Calculate the poset of the vertices ordered by inclusion of neighborhoods. By Corollary \ref{posetisom}, this poset must be isomorphic to the poset of their values ordered by divisibility. Therefore, the set of the minimal vertices is exactly the vertices of the set of all the primes dividing some element in $X$. Label them $p_1,p_2,...,p_k$.
    
    \item Given any vertex $v$ that is not minimal. $\{p_i: i\in I\}$ is the set of all vertices among $p_1,p_2,...,p_k$ that are smaller than $v$ ordered by inclusion of neighborhoods if and only if $v$ is the product of $\{p_i: i\in I\}$. So we have determined the value of each vertex in the contracted dual prime graph. (We do not know what primes are those $p_i$'s, but this does not affect the prime vertex graph up to isomorphism of graphs.)
    
    \item Check whether the set of vertices is closed under taking divisors. Also, check whether two vertices are adjacent if and only if they have nontrivial gcd. If either fails, reject the input graph.
    
    \item If the input graph is indeed some common divisor graph, we can now recover the prime vertex graph from the contracted dual prime graph with each vertex labeled by a product of distinct primes. The vertices of the prime vertex graph are exactly $\{p_1,p_2,\cdots p_k\}$ in Step 2. $p_ip_j$ is an edge in the prime graph if and only if $p_ip_j$ is a vertex in the contracted dual prime graph.
\end{enumerate}
\end{proof}

\begin{remark}
Notice that when we only consider $X=\eo(G)$ for some group $G$, $X$ is indeed closed under taking divisors. So if we are given a graph isomorphic to the dual prime graph $\Gamma(\eo(G))$, we can repeat the above algorithm to recover the prime graph $\Delta(\eo(G))$ of $G$. So we obtained a necessary condition for a graph to be isomorphic to the dual prime graph of some group. \bigskip

We can also consider $X=\eo(G)$ for some group $G$ of some certain class $\mathcal{C}$. If we have a necessary condition on the prime graphs of groups of class $\cC$, we can run the above algorithm and then check whether the resulted prime graph satisfies these conditions. This would give us a stronger necessary condition for a graph to be isomorphic to the dual prime graph of some group of class $\mathcal{C}$. In particular, we can take all classes of groups discussed in the previous sections.\bigskip

One may be able to further strengthen the above necessary condition by considering the possibility of adding back the non-square-free vertices. For example, if the square-free vertices are $\{p,q,pq\}$, then we cannot have that the clique containing $pq$ has 2 vertices while the clique containing $p$ and $q$ both have 1 vertex. This is because if $p^2q$ is an element order, so is $p^2$. However, it seems extremely complicated to answer when adding back the non-square-free vertices is possible.\bigskip

Still, we demonstrate the meaningfulness of the above necessary condition by providing the following example.
\end{remark}
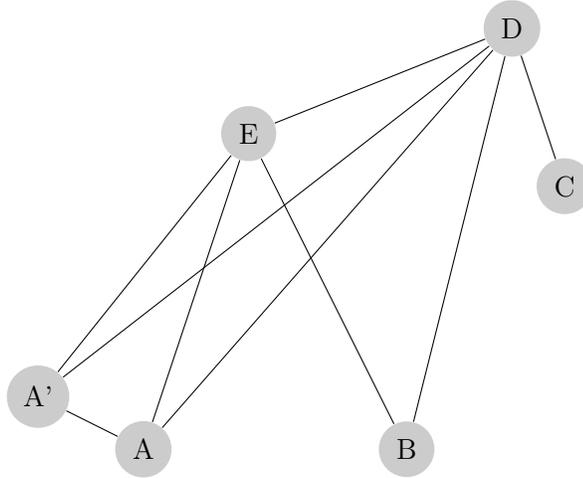
\begin{figure}
    \centering
\begin{tikzpicture}
  [scale=.7,auto=left,every node/.style={circle,fill=black!20}]
  \node (n0) at (2,6) {E};
  \node (n1) at (7,8)  {D};
  \node (n2) at (8,5)  {C};
  \node (n3) at (5,0) {B};
  \node (n4) at (-2,1)  {A'};
  \node (n5) at (0,0)  {A};
 
  \foreach \from/\to in {n0/n1, n0/n3, n0/n4, n0/n5, n1/n2, n1/n3, n1/n4, n1/n5, n4/n5}
    \draw (\from) -- (\to);
\end{tikzpicture}
\caption{A graph that is not isomorphic to the common divisor graph of any $X$ that is closed under taking divisors..}
\end{figure}

\begin{theoremN}
Let $F$ be the graph in Figure 2. $F$ is not isomorphic to the common divisor graph of any $X$ that is closed under taking divisors.
\end{theoremN}
\begin{proof}
We apply the algorithm.
\begin{enumerate}
    \item Notice that the only complete subgraph such that the neighborhoods of all its vertices are identical is $\{A,A'\}$. 
    \item After merging $A$ and $A'$, the poset structure of inclusion of neighborhoods is $D>E,A,B,C$ and $E>A,B$. So $A,B,C$ are minimal elements. Let the value of $A,B,C$ be $p,q,r$, respectively.
    \item Since $E$ is greater than $p$ and $q$, $E=pq$. Similarly, $D=pqr$.
    \item $pqr$ is in the graph, while $pr$ is not. Therefore, we reject $F$.
\end{enumerate}
\end{proof}

\section{Outlook}
In Section 6, we have Theorem \ref{weak result on C-groups} and Theorem \ref{second lemma on C-groups} that give a beautiful but uncomplete description of prime graphs of pseudo-solvable groups. It is natural to wonder what the conditions on the prime graphs would become if we allow, instead of $A_5$ as in the discussion of Section 6, other simple groups, such as $\PSL(2,7)$ or Suzuki groups, in the composition series. \bigskip

Theorem \ref{the algorithm}, the main theorem of Section 7, gives an algorithm that recovers the prime vertex graph from the common divisor graph and very likely rejects an invalid input. We wonder what criteria can be added so that our algorithm rejects all invalid inputs.

\section{Acknowledgements}
This research was conducted at Texas State University under NSF-REU grant DMS-1757233 during the summer of 2020. The fifth author gratefully acknowledges the financial support from the Office of Undergraduate Research at Washington University in St. Louis, and the other authors thankfully recognize NSF's support. The first three and the fifth author thank Texas State University for running the REU online during this difficult period of social distancing and providing a welcoming and supportive work environment. In particular, the sixth author, the director of the REU program, is recognized for conducting an inspired and successful research program. These authors also thank their mentor, the fourth author, for his invaluable advice and guidance throughout this project. The results in this paper were mainly discovered by the fifth author under the guidance of the fourth author, while the other members worked on two other papers.

%%\printbibliography

\end{document}